\documentclass[preprint,12pt]{elsarticle}

\usepackage{amssymb}
\usepackage{amsmath}
\usepackage{amsthm}

\newtheorem{theorem}{Theorem}
\newtheorem{lemma}{Lemma}
\newtheorem{definition}[theorem]{Definition}

\usepackage{lineno}

\newcommand{\Z}{\mathbb{Z}}
\newcommand{\Zp}{\mathbb{Z}/p\mathbb{Z}}

\journal{Journal}

\begin{document}

\begin{frontmatter}



\title{ONE  STEP  FURTHER  OF  AN  INVERSE  THEOREM  FOR THE  RESTRICTED  SET  ADDITION  IN $\Z/p\Z$}

\author{David Fernando Daza Urbano}
\author{René González-Martínez}


\author{Mario Huicochea Mason}
\author{Amanda Montejano Cantoral}

\begin{abstract}
Let $A$ and $B$ be sets of $k\ge5$ elements in $F=\Z/p\Z$ the field with $p>2k-2$ elements. We denote by $A\dot{+}B$ the set  of different elements of $F$ that can be written in the form $a+b$, where $a\in A$, $b\in B$, $a\neq b$. The number of elements of this set is at least $2k-3$.  K\'{a}rolyi showed that, except from some particular cases, The equality can only occur if $A = B$ and $A$ is an arithmetic progression with non zero difference. We prove that in the case that $|A\dot{+}B| = 2k - 2$ and $|A|=|B|$ the equality $A=B$ holds.
\end{abstract}

\end{frontmatter}


\section{Preliminars}
\label{sec:preliminars}
    Let $p$ be a prime, and let $A$ and $B$ be subsets of the field $F=\Z/p\Z$ we consider the sets
    \[
        A+B=\{a+b|a \in A, b \in B\}
    \]
    called the \emph{sumset} of $A$ and $B$, and 
    \[
        A\dot{+}B=\{a+b|a \in A, b \in B \text{ and }a\neq b\}
    \]
    called the \emph{restricted sumset} of $A$ and $B$.
    
    The Cauchy-Davenport Theorem \cite{Cauchy,Davenport} states that if $p\ge k+\ell -1$, then $|A+B|\ge k+\ell-1$ holds for every $A$ and $B$ where $k=|A|$ and $\ell=|B|$. In $1996$ Alon, Nathanson, and Ruzsa provided a simple proof of this theorem using the Polynomial Method \cite{Alon}.

    Gyula K\'{a}rolyi extended this result for abelian groups. To state this result, first we introduce some notation.

    \begin{definition}
        Let $G$ be a group.  We define $p(G)$ to be the smallest positive integer $p$ for which there exists a nonzero element $g$ of $G$ with $pg=0$. If no such $p$ exists, we write $p(G)=\infty$.
    \end{definition}
    Using the same notation as before.
    
    \begin{theorem}[K\'{a}rolyi \cite{Gyula1,Gyula2}]
        If $A$ and $B$ are nonempty subsets of an abelian group $G$ such that $p(G)\ge k+\ell-1$, then $|A+B|\ge k+\ell-1$.
    \end{theorem}

    It's natural to ask the inverse problem. Namely, which conditions should be imposed on $A$ and $B$ such that the equality in the Cauchy-Davenport holds?

    We say that a pair of subsets $A$ and $B$ are critical if $A+B \neq \Z/p\Z$ and $|A+B|<|A|+|B|$. Vosper completely classified all the critical pairs \cite{Vosper}.

    \begin{theorem}[Vosper]
        Let $A$ and $B$ be subsets of $F=\Zp$ such that $|A+B|<|F|$. Then
        \[
            |A+B| = |A|+|B|-1
        \]
        if and only if at least one of the following holds:
        \begin{itemize}
            \item $\min(|A|,|B|)=1$,
            \item $|A+B|=p-1$ and $B=F\backslash (c - A)$, where $c=F\backslash(A+B)$ 
            \item $A$ and $B$ are arithmetic progressions with the same common difference.
        \end{itemize}
    \end{theorem}

    Hamidoune and Rødseth went one step further by characterizing the critical pairs of Vosper's Theorem.

    \begin{theorem}
        Let $A$ and $B$ be subsets of $\Zp$, with $|A|\ge3$ and $|B|\ge4$ such that
        \[
            |A+B|=|A|+|B|\le p-4.
        \]
        Then $A$ and $B$ are contained in arithmetic progressions with the same difference of respective lenghts $|A|+1$ and $|B|+1$.
    \end{theorem}

    Closely related to the Cauchy-Davenport Theorem there is conjecture Paul Erd\H{o}s and Hans Heilbronn posed in the early 1960s for the restricted sumset. Namely, the lower bound changes only slightly \cite{Erdos Col. Con.,Erdos, Erdos Graham}. 

    \begin{theorem}[Erd\H{o}s-Heilbronn Problem]
    If $A$ and $B$ are non-empty subsets of $F=\Z/p\Z$ with $p$ prime, and $p\ge k+\ell-3$ then $|A\dot{+}B|\ge k+\ell-3\}$
    \end{theorem}

    In 1994 Dias da Silva and Hamidoune \cite{Dias da Silva} proved that for the special case $A=B$
    \[
        |A\dot{+}B|\ge 2k-3
    \]
    using exterior algebra methods. Their proof for the general case was established later by Alon, Nathanson and Ruzsa \cite{Alon} using the so called \emph{polynomial method}.

    \begin{lemma}[Combinatorial Nullstellensatz]\label{CombNull}
        Let $F$ be an arbitrary field and let $f=f(x_1,\ldots,x_n)$ be a polynomial in $F[x_1,\ldots,x_n]$. Let $S_1,\ldots,S_n$ be nonempty subsets of $F$ and define $g_i(x_i)=\prod_{s\in S_i}(x_i-s)$. If $f(s_1,\ldots,s_n)=0$ for all $s_i\in S_i$, the there exist polynomials $h_i(x_1,\ldots,x_n)\in F[x_1,\ldots,x_n]$ for every $i\in\{1,\ldots,n\}$ satisfying $\deg(h_i)\le\deg(f)-\deg(g_i)$ such that
        \[
            f=\sum_{i=1}^{n}h_i(x_1,\ldots,x_n)g_i(x_i).
        \]
    \end{lemma}

    Again, we can ask for inverse theorems, Gyula Károlyi proved the following theorem \cite{Gyula3}.

    \begin{theorem}
        Let $A$ and $B$ subsets of $F=\Zp$ with $p>|A|+|B|-3$, $|A|\ge5$ and $|B|\ge5$. Then, $|A+B|=|A|+|B|-3$ if and only if $A=B$ and $A$ is an arithmetic progression.
    \end{theorem}

\section{Restricted sumsets: The case $|A|=|B|=k$}
    For any subset $X\subseteq F=\Z/p\Z$ we denote by $\sigma_i(X)$ with $1\le i \le|X|$ to the $i$th symmetrical polynomial of $X=\{x_1,x_2,\ldots,x_{|X|}\}$ and $\sigma_0(X)=1$. We use the same strategy proposed by Alon et. al. to prove the following.
\begin{theorem}
    Let $A$ and $B$ be two subsets of $F=\Z/{p\Z}$ with $|A|=|B|=k$ such that $|A\dot{+}B| = 2k-2$. Then $A=B$.
\end{theorem}
\begin{proof}
    Let $A=\{a_1,\ldots,a_k\}$, $B=\{b_1,\ldots,b_k\}$ and $C=A\dot{+}B=\{c_1,\ldots,c_{2k-2}\}$ and let
    \[
        f(x,y)=(x-y)\prod_{j=1}^{2k-2}(x+y-c_j)\in F[x,y]
    \]
    and
    \[
        g_A(z)=\prod_{z=1}^{k}(z-a_i),\,g_B(z)=\prod_{z=1}^k(z-b_i)\in F[z].
    \]

    We can write
    \[
        f(x,y) =  \sum_{i=0}^{2k-1}f_{2k-2-i}(x,y)
    \]
    as the sum of its homogeneous components. We prove by induction on $i$ that the coefficients of $g_A(z)$ and $g_B(z)$ are equal, and consequently the roots of both polynomials must be the same.
    
    For every $x\in A$ and $y\in B$ we have that $f(x,y)=0$. Hence, by Lemma \ref{CombNull}, there exist polynomials $h_A(x,y)$ and $h_B(x,y)$ such that 
    \begin{equation}\label{eqCN}
        f(x,y) = h_A(x,y)g_A(x)+h_B(x,y)g_B(y),
    \end{equation}
    with $\deg(h_A(x,y))\le k-1$ and $\deg(h_B(x,y))\le k-1$. Decomposing $h_A$ and $h_B$ into its homogeneous components we can write
    \[
        h_A(x,y)=\sum_{i=0}^{k-1} h_{A,i}(x,y), \,\text{where} \, h_{A,i} (x,y) = \sum_{j=0}^{i}A_{ij}x^jy^{i-j}
    \]
    and
    \[
        h_B(x,y)=\sum_{i=0}^{k-1} h_{B,i}(x,y), \,\text{where} \, h_{B,i} (x,y) = \sum_{j=0}^{i}B_{ij}x^{i-j}y^{j}.
    \]
    We can also decompose $f(x,y)$ into its homogeneous components, we can rewrite $f(x,y)$ as
    \begin{equation}\label{homdecomp}
        f(x,y) = \sum_{i=0}^{2k-2}(-1)^{i}\sigma_{i}(C)p_{2k-1-i}(x,y),
    \end{equation}
    here 
    \[
        p_i(x,y)=(x-y)(x+y)^{i-1}=\sum_{j=0}^{i} C_{ij}x^jy^{i-j}
    \]
    where $C_{ii}=1$, $C_{i0}=-1$ and otherwise
    \[
        C_{ij}={\binom{i-1}{j-1}}-{\binom{i-1}{j}}=\frac{2j-i}{j}{\binom{i-1}{j-1}}=\frac{2j-i}{j}{\binom{i-1}{i-j}},
    \]
    in particular we note that $C_{ij}=-C_{i,(i-j)}$.
    
    Taking the homogeneous component of $f(x,y)$ of degree $2k-1$ and by equations \ref{eqCN} and \ref{homdecomp} we have that
    \begin{align*}
        f_{2k-1}(x,y)   &=p_{2k-1}(x,y)\\
                        &=\sum_{j=0}^{2k-1}C_{2k-1,j}x^jy^{2k-1-j}\\
                        &=h_{A,k-1}(x,y)x^k+h_{B,k-1}(x,y)y^k 
    \end{align*}
    by taking the coefficients of the monomial $x^{k+i}y^{k-i-1}$ and $x^{k-i-1}y^{k+i}$ with $0\le i\le k-1$ we conclude that
    \[
        A_{k-1,i}=C_{2k-1,k+i},
    \]
    and
    \begin{equation}\label{Ak-1}
        B_{k-1,i}=-A_{k-1,i}=C_{2k-1,k-i-1}.
    \end{equation}
    these coefficients are not zero by our assumption $p > 2k - 2$.
    
    We look at the homogeneous component of $f(x,y)$ of degree $2k-2$. That is 
    \begin{align*}
        f_{2k-2}(x,y)   &=h_{A,k-2}(x,y)x^k-\sigma_1(A)h_{A,k-1}(x,y)x^{k-1}\\
                        &+h_{B,k-2}(x,y)y^k-\sigma_1(B)h_{B,k-1}(x,y)y^{k-1}.
    \end{align*}
    First we look at the coefficient of the monomial $x^{k-1}y^{k-1}$. That is
    \[
        C_{2k-2,k-1}=0=-\sigma_1(A)A_{k-1,0}-\sigma_1(B)B_{k-1,0}.
    \]
    By Equation \ref{Ak-1} we conclude that $\sigma_1(A)=\sigma_1(B)$.

    Next by inspecting the coefficients of the monomials $x^{2k-2-j}y^j$ and $x^jy^{2k-2-j}$ with $1\le j \le k-2$ we conclude that
    \[
        C_{2k-2,2k-2-j}=A_{k-2,k-2-j}-\sigma_1(A)A_{k-1,k-1-j}
    \]
    and
    \[
        C_{2k-2,j}=B_{k-2,k-2-j}-\sigma_1(B)B_{k-1,k-1-j}.
    \]
    Since $C_{2k-2,2k-2-j}=-C_{2k-2,j}$ and $\sigma_1(A)A_{k-1,k-1-j}=-\sigma_1(B)B_{k-1,k-1-j}$ then $A_{k-2,k-2-j}=-B_{k-2,k-2-j}$.

    Now, we assume that for every $r<i$ with $0\le i\le k$ $\sigma_r(A)=\sigma_r(B)$ and $A_{k-1-r,j}=-B_{k-1-r,j}$. By taking the homogeneous component of $f$ of degree $2k-1-i$
    \begin{align*}
        f_{2k-1-i}(x,y) &=\sum_{j=0}^{i}(-1)^j \sigma_j(A)h_{A,k-1-i+j}(x,y)x^{k-j}\\
                        &+\sum_{j=0}^{i}(-1)^j \sigma_j(B)h_{B,k-1-i+j}(x,y)y^{k-j}.
    \end{align*}
    If $i$ is even $i=2r$. In this case, we look at the coefficients of the monomials $x^{k-r-1}y^{k-r}$ and $x^{k-r}y^{k-r-1}$.
    \begin{align*}
        \sigma_{2r}(C)C_{2k-2r-1,k-r-1} &=\sum_{j=0}^{r}(-1)^{r+j}\sigma_{r+j}(B)B_{k-r-1+j,j}\\
                                        &+\sum_{j=0}^{r-1}(-1)^{r+1+j}\sigma_{r+1+j}(A)A_{k-r+j,j}
    \end{align*}
    
    \begin{align*}
        \sigma_{2r}(C)C_{2k-2r-1,k-r}   &=\sum_{j=0}^{r}(-1)^{r+j}\sigma_{r+j}(A)A_{k-r-1+j,j}\\
                                        &+\sum_{j=0}^{r-1}(-1)^{r+1+j}\sigma_{r+1+j}(B)B_{k-r+j,j}.
    \end{align*}
    It follows from the induction hypothesis that 
    \[
        \rho_i = \sigma_i(B)B_{k-1,r} - \sigma_i(A)B_{k-1,r-1} 
    \]
    and
    \[
        -\rho_i = -\sigma_i(A)B_{k-1,r} + \sigma_i(B)B_{k-1,r-1} 
    \]
    where
    \begin{align*}
        \rho_i = \sigma_{2r}(C)C_{2k-2r-1,k-r-1} &- \sum_{j=0}^{r-1}(-1)^{r+j}\sigma_{r+j}(B)B_{k-r-1+j,j}\\ 
        &-\sum_{j=0}^{r-2}(-1)^{r+1+j}\sigma_{r+1+j}A_{k-r+j,j}.
    \end{align*}
    Therefore
    \[
        0 = (\sigma_i(B)-\sigma_i(A))(B_{k-1,r-1}+B_{k-1,r}),
    \]
    but
    \begin{align*}
        B_{k-1,r-1}+B_{k-1,r}   &=C_{2k-1,k-r}+C_{2k-1,k-r-1}\\
                                &=\frac{-2r(2k-1)(2k-2)!}{(k+r)!(k-r)!}\\
                                &\neq 0
    \end{align*}
    since $p>2k-1$. Hence
    \[
        \sigma_i(A) = \sigma_i(B).
    \]
    Similarly, if $i$ is odd, $i=2r+1\le k$. The coefficient of the monomial $x^{k-r-1}y^{k-r-1}$ is
    \[
        0=\sum_{j=0}^{r}(-1)^{r+1+j}\left(\sigma_{r+1+j}(A)A_{k-r-1+j,j} + \sigma_{r+1+j}(B)B_{k-r-1+j,j}\right)
    \]
    since $C_{2k-2r-2,k-r-1}=0$. Again, by our induction hypothesis $\sigma_{r+1+j}(A)=\sigma_{r+1+j}(B)$ for every $0\le j < r$ and $A_{k-r-1+j,j}=-B_{k-r-1+j,j}$ for every $0\le j \le r$. Hence 
    \[
        0=(\sigma_{i}(B)-\sigma_{i}(A))B_{k-1,r}.
    \]
    Again 
    \[
        B_{k-1,r}=\frac{-i}{k-r-1}\binom{2k-2}{k+r-1}\neq0.
    \]
    
    Now,  we inspect the coefficients of the monomials $x^{k+\ell}y^{k-i-1-\ell}$ and $x^{k-i-1-\ell}y^{k+\ell}$ for $2\le \ell\le k-i-1$. That is,
    \[
        \sigma_i(C)C_{2k-i-1,k+\ell}=\sum_{j=0}^{2r}(-1)^j\sigma_j(A)A_{k-i-1+j,\ell+j},
    \]
    and
    \[
        \sigma_i(C)C_{2k-i-1,k-i-1-\ell}=\sum_{j=0}^{2r}(-1)^j\sigma_j(B)B_{k-i-1+j,\ell+j}.
    \]
    Since $C_{2k-i-1,k+\ell}=-C_{2k-i-1,k-i-1-\ell}$, and by our induction hypothesis \[
        A_{k-i-1+j,\ell+j}=-B_{k-i-1+j,\ell+j}
    \]
    for $j\ge1$. We conclude that $A_{k-i-1,\ell}=-B_{k-i-1,\ell}$.

\end{proof}

 \bibliographystyle{elsarticle-num} 
 
\end{document}